\documentclass[12pt]{amsart}
\usepackage{amsfonts}
\usepackage{amsmath}
\usepackage{amssymb}
\usepackage[margin=1.2in]{geometry}
\newtheorem{theorem}{Theorem}%[section]
\newtheorem{corollary}{Corollary}
\newtheorem{lemma}{Lemma}
\newtheorem{proposition}{Proposition}

\theoremstyle{definition}

\newtheorem{remark}{Remark}
\newtheorem{example}{Example}

\numberwithin{equation}{section}

\begin{document}
\title[On PNT equivalences for Beurling numbers]{On PNT equivalences for Beurling numbers}

\author[G. Debruyne]{Gregory Debruyne}
\thanks{G. Debruyne gratefully acknowledges support by Ghent University, through a BOF Ph.D. grant}
\address{G. Debruyne\\ Department of Mathematics\\ Ghent University\\ Krijgslaan 281\\ B 9000 Gent\\ Belgium}
\email{gdbruyne@cage.UGent.be}
\author[J. Vindas]{Jasson Vindas} \thanks{The work of J. Vindas was supported by the Research Foundation--Flanders, through the FWO-grant number 1520515N}
\address{J. Vindas\\ Department of Mathematics\\ Ghent University\\ Krijgslaan 281\\ B 9000 Gent\\ Belgium}
\email{jvindas@cage.UGent.be}
\subjclass[2010]{Primary11N80; Secondary 11M41}
\keywords{Beurling generalized numbers; sharp Mertens relation; prime number theorem; mean-value vanishing of the M\"{o}bius function; zeta functions; Landau relations; PNT equivalences}

\begin{abstract}
In classical prime number theory several asymptotic relations are considered to be ``equivalent" to the prime number theorem. In the setting of Beurling generalized numbers, this may no longer be the case. Under additional hypotheses on the generalized integer counting function, one can however still deduce various equivalences between the Beurling analogues of the classical PNT relations. We establish some of the equivalences under weaker  conditions than were known so far.
\end{abstract}

\maketitle

\section{Introduction}

Several asymptotic relations in classical prime number theory are considered to be ``equivalent" to the prime number theorem. This means that they are deducible from one another by simple real variable arguments (see \cite[Sect. 5.2]{bateman-diamondbook}, \cite{diamond1982}, and \cite[Sect. 6.2]{na}). In recent works \cite{diamond-zhang,diamond-zhangbook}, Diamond and Zhang have investigated the counterparts of several of these classical asymptotic relations in the context of Beurling generalized numbers. They showed by means of examples that some of the implications between the relations may fail without extra hypotheses, and they found conditions under which the equivalences do or do not hold. 

The aim of this article is to improve various of their results by relaxing hypotheses on the generalized number systems. While Diamond and Zhang employed elementary methods in \cite{diamond-zhang,diamond-zhangbook} (a version of Axer's lemma and convolution calculus for measures), our approach here is different. Our arguments are based on recent complex Tauberian theorems for Laplace transforms with pseudofunction boundary behavior \cite{debruyne-vindasCT,korevaar2005}. This approach will enable us to clarify that only certain boundary properties of the zeta function near $s=1$ play a role for the equivalences.  

Let us introduce some terminology in order to explain our results. A Beurling generalized prime number system \cite{bateman-diamond,beurling,diamond-zhangbook} is simply an unbounded sequence of real numbers $p_{1}\leq p_{2}\leq p_3\leq  \dots$ with the only requirement $p_{1}>1$. The set of generalized integers is the multiplicative semigroup generated by the generalized primes and 1. We arrange them in a non-decreasing sequence where multiplicities are taken into account, $1=n_0<n_{1}\leq n_{2}\leq \dots $. One then considers the counting functions
\begin{equation*}
N(x)=\sum_{n_{k}\leq x}1, \quad \pi(x)=\sum_{p_{k}\leq x}1\ , \quad \Pi(x)=\pi(x)+\frac{1}{2}\pi(x^{1/2})+\frac{1}{3}\pi(x^{1/3})+\dots\: ,
\end{equation*}
and (the Chebyshev function)
\begin{equation}                                        
\label{ibpneq1}
\psi(x):=\int_{1}^{x}\log t \: \mathrm{d}\Pi(t)=\sum_{n_{k}\leq x}\Lambda(n_k).
\end{equation}
 As in classical number theory, the PNT $\pi(x)\sim x/\log x$ always becomes equivalent \cite{bateman-diamond} to $\Pi(x)\sim x/\log x$, and to 
\begin{equation*}
\psi(x)\sim x.
\end{equation*}
 
We are also interested in the asymptotic relation
 \begin{equation}                                        
\label{ibpneq3}
\psi_{1}(x) := \int^{x}_{1} \frac{\mathrm{d}\psi(t)}{t} = \sum_{n_{k} \leq x} \frac{\Lambda(n_{k})}{n_{k}} = \log x + c + o(1),
\end{equation}
which, as in \cite{diamond-zhang}, we call a \emph{sharp Mertens relation}. Note that for the ordinary rational primes (\ref{ibpneq3}) holds with $c=-\gamma$, where $\gamma$ is the Euler-Mascheroni constant; however, in general, we may have $c\neq-\gamma$. (For instance, adding an extra prime to the rational primes makes $c>-\gamma$.) Nonetheless, $c$ may be related to a generalized gamma constant associated to the generalized number system, see (\ref{eqconst-gamma}) below. 

The sharp Mertens relation is known to be equivalent to the PNT for rational primes. In the general case it is very easy to see that (\ref{ibpneq3}) always yields the PNT \cite[Prop. 2.1]{diamond-zhang} for Beurling primes. On the other hand, it was shown in \cite{diamond-zhang} that the converse implication only holds conditionally. 

Our first goal is to investigate conditions under which the equivalence between the PNT and the sharp Mertens relation remains true. In particular, we shall show:

\begin{theorem} \label{thmainsimplified} Suppose that a generalized number system satisfies the PNT and the conditions
	\begin{equation}
	\label{thmainas1simplified}
	N(x) = ax + o\left(\frac{x}{\log x}\right)
	\end{equation}
	and 
	\begin{equation}
	\label{thmainas3simplified} 
	\int^{\infty}_{1} \left|\frac{N(x) - ax}{x^{2}}\right| \mathrm{d}x < \infty,          
	\end{equation} 
	for some $a > 0$. Then, the sharp Mertens relation (\ref{ibpneq3}) is satisfied as well with constant 
\begin{equation}
	\label{eqconst-gamma} 	 
	c =-1 -\frac{1}{a}\int^{\infty}_{1} \frac{N(x) - ax}{x^{2}} \mathrm{d}x= -\frac{1}{a}\lim_{x\to\infty} \left(\sum_{n_{k}\leq x}\frac{1}{n_{k}}-a \log x\right).
	\end{equation}
\end{theorem}

Theorem \ref{thmainsimplified} contains the following result of Diamond and Zhang:

\begin{corollary}[\cite{diamond-zhang}]\label{cdiamond-zhang} Suppose that the PNT holds and for some $a>0$
\begin{equation}
\label{diamond-zhangcond1}
\left|\frac{N(x) - ax}{x}\right|\leq D(x), \quad x\geq1,
\end{equation}
where $D$ is right continuous, non-increasing, and satisfies
\begin{equation}
\label{diamond-zhangcond2}
\int_{1}^{\infty}\frac{D(x)}{x}\mathrm{d}x<\infty.
\end{equation}
Then, (\ref{ibpneq3}) holds.
\end{corollary}
\begin{proof}
Clearly the assumption implies (\ref{thmainas3simplified}). Moreover, since $D$ is non-increasing, we must have $D(x)=o(1/\log x)$, so (\ref{thmainas1simplified}) should hold as well.
\end{proof}

A simple condition that is included in those of both Theorem \ref{thmainsimplified} and Corollary \ref{cdiamond-zhang} is  
\begin{equation}
\label{diamond-zhangcond3}
N(x) = ax + O\left(\frac{x}{\log^{\alpha} x}\right)
\end{equation}
if $\alpha>1$. Interestingly, (\ref{diamond-zhangcond3}) with $\alpha=1$ and the PNT are not strong enough to ensure the sharp Mertens relation, as established by an example in \cite{diamond-zhang,diamond-zhangbook}. We will strengthen that result as well.

\begin{proposition}
\label{prop1pnteex} We have:
\begin{itemize}
\item[(i)] The PNT and (\ref{thmainas1simplified}) do not necessarily imply the sharp Mertens relation.
\item[(ii)] The PNT and (\ref{thmainas3simplified}) do not necessarily imply the sharp Mertens relation either.
\end{itemize}
\end{proposition}

Under the hypotheses (\ref{diamond-zhangcond1}) and (\ref{diamond-zhangcond2}) with $D$ as in Corollary \ref{cdiamond-zhang}, Diamond and Zhang were also able to show \cite{diamond-zhangbook} the equivalence between
\begin{equation} \label{eqlandaurelation}
 M(x) := \sum_{n_{k}\leq x} \mu(n_{k}) = o(x)
\end{equation}
and
\begin{equation} \label{eqlandausharprelation}
 m(x) := \sum_{n_{k}\leq x} \frac{\mu(n_{k})}{n_{k}} = o(1),
\end{equation}
with $\mu$ the Beurling analogue of the M\"obius function. We will also improve this result by using a weaker condition. Note that for rational primes the equivalence between (\ref{eqlandaurelation}), (\ref{eqlandausharprelation}), and the PNT was first established by Landau in 1911 (\cite{landau}, \cite[Sect. 6.2.7]{na}); because of that, we refer to them as the \emph{Landau relations}. It is worth noticing that the implication (\ref{eqlandausharprelation}) $\Rightarrow$ (\ref{eqlandaurelation}) holds unconditionally, as can easily be seen via integration by parts; 
 therefore, one only has to focus on the conditional converse.

\begin{theorem} \label{thmainmobius} (\ref{eqlandaurelation}) and the condition (\ref{thmainas3simplified}) for some $a>0$ imply the other Landau relation (\ref{eqlandausharprelation}).
\end{theorem}

We point out that we have only stated here our main results in their simplest forms. In Section \ref{section sharp Mertens} and Section \ref{sectionotherequivalences} we will replace (\ref{thmainas1simplified}) and (\ref{thmainas3simplified}) by much weaker assumptions in terms of convolution averages. These convolution average versions of (\ref{thmainas1simplified}) and (\ref{thmainas3simplified}) express the fact that only the local behavior of the zeta function
$$
\zeta(s)=\sum_{k=0}^{\infty}\frac{1}{n^{s}_{k}}
$$
at $s=1$ is responsible for the equivalences under consideration. In Section \ref{pnte examples} we construct examples in order to give a proof of Proposition \ref{prop1pnteex}.

\section{Tauberian tools}
Our analysis in the next sections makes extensive use of complex Tauberian theorems for Laplace transforms, which we collect here together with some background material on related concepts for the reader's convenience. These Tauberian theorems are in terms of local pseudofunction boundary behavior \cite{debruyne-vindasCT,katznelson,korevaarbook,korevaar2005,s-v}, which turns out to be an optimal assumption on the Laplace transform, in the sense that it often leads to ``if and only if''  results. See also \cite{diamond-zhangbook,zhang2014} for an $L^{1}_{loc}$-approach to ``if and only if'' type complex Tauberian theorems. Additionally, we will also employ the Wiener division type theorem for non-quasianalytic Beurling algebras; we refer to \cite{beurling1938,korevaarbook} for the latter topic. 

Pseudofunctions are a special kind of  Schwartz distributions that arise in harmonic analysis \cite{benedettobook} and are defined via Fourier transform. The standard Schwartz test function spaces are denoted by $\mathcal{D}(\mathbb{R})$ and $\mathcal{S}(\mathbb{R})$, while $\mathcal{D}'(\mathbb{R})$ and $\mathcal{S}'(\mathbb{R})$ stand for their topological duals, the spaces of distributions and tempered distributions. The Fourier transform, normalized as $\hat{\varphi}(t)=\int_{-\infty}^{\infty}e^{-itx}\varphi(x)\:\mathrm{d}x,$ is defined as usual on $\mathcal{S}'(\mathbb{R})$ via duality. It is important to notice that if $f\in\mathcal{S}'(\mathbb{R})$ has support in $[0,\infty)$, its Laplace transform 
$\mathcal{L}\left\{f;s\right\}=\left\langle f(u),e^{-su}\right\rangle$ is well-defined, analytic on $\Re e\:s>0$, and has distributional boundary value $\hat{f}$ on $\Re e\:s=0$. See the textbooks \cite{schwartz,vladimirov} for distribution theory and \cite{estrada-kanwal,p-s-v} for asymptotic calculus with distributions. 

A tempered distribution $f\in\mathcal{S}'(\mathbb{R})$ is called a (global) \emph{pseudofunction} if $f=\hat{g}$, where $g\in L^{\infty}(\mathbb{R})$ and $\lim_{|x|\to\infty}g(x)=0$. We denote the space of pseudofunctions by $PF(\mathbb{R})$.
Given an open interval $I$, we say that a distribution $f$ is a \emph{local} pseudofunction on $I\subset\mathbb{R}$ if $f$ coincides with a pseudofunction on a neighborhood of each point of $I$. We write $f\in PF_{loc}(I)$. One can easily check \cite{korevaar2005} that $f\in PF_{loc}(I)$ if and only if
\begin{equation}
\label{GRLeq}
\widehat{\varphi f}(h)=\langle f(t), e^{-iht} \varphi(t) \rangle=o(1), \quad |h|\to\infty, \quad \mbox{for each }\varphi\in\mathcal{D}(I).
\end{equation}
The property (\ref{GRLeq}) can be regarded as a generalized Riemann-Lebesgue lemma.
In particular, $L^{1}_{loc}(I)\subset PF_{loc}(I)$ in view of the classical Riemann-Lebesgue lemma. 

Let $F(s)$ be analytic on the half-plane $\Re e\:s>\alpha$. We say that $F$ has local pseudofunction boundary behavior on the boundary (open) line segment $\alpha+iI$ if there is $f\in PF_{loc}(I)$ such that
\begin{equation}
\label{bveq}
\lim_{\sigma\to\alpha^{+}}\int_{-\infty}^{\infty}F(\sigma+it)\varphi(t)\mathrm{d}t=\left\langle f(t),\varphi(t)\right\rangle\ , \quad \mbox{for each } \varphi\in\mathcal{D}(I).
\end{equation}
Boundary behavior with respect to other distribution subspaces is defined analogously. We write in short $F(\alpha+it)=f(t)$ for boundary distributions in the sense of (\ref{bveq}). We emphasize that $L^1_{loc}$, continuous, or analytic extension are very special cases of local pseudofunction boundary behavior.

We call a function $\tau$ \emph{slowly decreasing} if for each $\varepsilon > 0$ there exists $\eta>0$ such that 
\[  \liminf_{x\to\infty}\inf_{h\in[0,\eta]}(\tau(x+h) - \tau(x)) > - \varepsilon.
\]
Note that every non-decreasing function is slowly decreasing.
The first Tauberian theorem that we state is a recent extension of the Ingham-Fatou-Riesz theorem \cite{ingham1935}, obtained by the authors in \cite{debruyne-vindasCT}.

\begin{theorem} \label{thtauberian1} Let $\tau\in L^{1}_{loc}(\mathbb{R})$ be slowly decreasing with $\operatorname*{supp}\tau\subseteq[0,\infty)$. Then,
$$\tau(x) =ax+b + o(1)$$ if and only if its Laplace transform converges for $\Re e \: s > 0$ and
\begin{equation*}
\mathcal{L}\{\tau;s\}  - \frac{a}{s^{2}} -\frac{b}{s}  
\end{equation*}
admits local pseudofunction boundary behavior on the line $\Re e \: s = 0$.   
\end{theorem}  
We point out the ``if'' direction in Theorem \ref{thtauberian1} requires the Tauberian condition that $\tau$ is slowly decreasing, while the ``only if'' part is of Abelian character and does not require the Tauberian condition.

We will also employ the following distributional version of the Wiener-Ikehara theorem, due to Korevaar \cite{korevaar2005}. (It can be deduced from Theorem \ref{thtauberian1}, see \cite[Thm.~5.4]{debruyne-vindasCT}.)

\begin{theorem} \label{thtauberian2} Let $S$ be a non-decreasing function having support in $[0,\infty)$. Then, 
$$S(x)\sim ae^{x}$$
 if and only if $\mathcal{L}\{\mathrm{d}S;s\} = \int^{\infty}_{0^{-}} e^{-sx} \mathrm{d}S(x)$ converges for $\Re e \: s > 1$
and 
\begin{equation*}
  \mathcal{L}\{\mathrm{d}S;s\}- \frac{a}{s-1}
\end{equation*}
 admits local pseudofunction boundary behavior on the line $\Re e \: s = 1$.   
\end{theorem} 
Similarly as for Theorem \ref{thtauberian1}, one does not need that $S$ is non-decreasing for the ``only if'' Abelian direction of Theorem \ref{thtauberian2}. 

The ensuing lemma is very useful to conclude an $O(e^{x})$-bound from the local behavior of the Laplace transform just near $s=1$; its proof is simple, see \cite[Thm. 3.1]{debruyne-vindasCT}. It is in terms of local pseudomeasure boundary behavior. The space of (global) pseudomeasures is $PM(\mathbb{R})=\mathcal{F}(L^{\infty}(\mathbb{R}))$ and we of course have the inclusion $PF(\mathbb{R})\subset PM(\mathbb{R})$; we define local pseudomeasures and local pseudomeasure boundary behavior in the same way as it has been done for the pseudofunction analogues. We use the notation $PM_{loc}(I)$ for the space of local pseudomeasures on an open interval.

\begin{lemma}
\label{thtauberian3}
Let $S$ be non-decreasing, have support in $[0,\infty)$, and have convergent Laplace-Stieltjes transform for $\Re e \: s > 1$. If $\mathcal{L}\{\mathrm{d}S;s\}$ has local pseudomeasure boundary behavior on a line segment of $\Re e\: s=1$ containing the point $s=1$, then $S(x)\ll e^{x}.$
\end{lemma}

We also need to discuss multipliers for local pseudofunctions.
From the generalized Riemann-Lebesgue lemma (\ref{GRLeq}), it is already clear that smooth functions are multipliers for the space $PF_{loc}(I)$. More general multipliers can be found if we employ the Wiener algebra $A(\mathbb{R})$(=$\mathcal{F}(L^{1}(\mathbb{R}))$). In fact, the multiplication of $f\in PF(\mathbb{R})$ with $g\in A(\mathbb{R})$ can be canonically defined by convolving in the Fourier side and then taking inverse Fourier transform; we obviously have $fg\in PF(\mathbb{R})$. By going to localizations (and gluing then with partitions of the unity), the multiplication $fg\in PF_{loc}(I)$ can be extended for $f\in PF_{loc}(I)$ and $g\in A_{loc}(I)$, where the latter membership relation means that $\varphi g\in A(\mathbb{R})$ for all $\varphi\in \mathcal{D}(I)$. 

The next lemma deals with a simple situation in which the just defined multiplication agrees with the operation of taking distributional boundary values of the product of two Laplace transforms. The space $\mathcal{B}'(\mathbb{R})$ denotes the Schwartz space of bounded distributions, it consists of the linear span of all (distributional) derivatives of $L^{\infty}$-functions; likewise, $\mathcal{D}'_{L^1}(\mathbb{R})$ stands for the Schwartz space of integrable distributions, spanned by derivatives of $L^1$-functions. In particular, $L^{\infty}(\mathbb{R})\subset \mathcal{B}'(\mathbb{R})$ and $L^{1}(\mathbb{R})\subset \mathcal{D}'_{L^1}(\mathbb{R})$. Properties of these spaces are discussed in detail in Schwartz' book \cite[Sect. VI.8]{schwartz} (see also \cite[Sect. 5]{d-p-v1}).
\begin{lemma}\label{productlpfbb}
Let $F(s)=\mathcal{L}\{f;s\}$ and $G(s)=\mathcal{L}\{g;s\}$, where $f,g\in\mathcal{S}'(\mathbb{R})$ have supports in $[0,\infty)$. Suppose that $F(s)$ has local pseudofunction boundary behavior on the line segment $iI$ of $\Re e\:s=0$ and $G$ has $A_{loc}$-boundary behavior on $iI$. If either $f\in \mathcal{B}'(\mathbb{R})$ or $g\in \mathcal{D}'_{L^{1}}(\mathbb{R})$, then $F(s)G(s)$ has local pseudofunction boundary behavior on $iI$, namely, the local pseudofunction $F(it)\cdot G(it)\in PF_{loc}(I)$. 
\end{lemma} 
\begin{proof} Naturally, $F(s)G(s)$ has distributional boundary value $\widehat {f\ast g}$; what we have to verify is that if $\varphi \in\mathcal{S}(\mathbb{R})$ is such that $\hat{\varphi}\in\mathcal{D}(I)$, then $f\ast g\ast\varphi \in L^{\infty}(\mathbb{R})$ and tends to 0 at $\pm\infty$. Now, we can factor $\varphi=\varphi \ast \phi $ with $\hat{\phi}\in \mathcal{D}(I)$ being equal to 1 in a neighborhood of $\operatorname*{supp}\hat{\varphi}$. Note that in general the convolution of distributions may not be associative; however, in this case our hypotheses guarantee the conditions from \cite[Sect. 4.2.8, p. 56]{vladimirov} for associativity. So, the smooth function $(f\ast g\ast\varphi)(h)=((f\ast \varphi )\ast (g\ast \phi))(h)\to0$ as $|h|\to\infty$, since $g\ast\phi \in L^{1}(\mathbb{R})$ and $\lim_{|x|\to\infty}(f\ast \varphi)(x)=0$. Furthermore, $F(s)G(s)$ has boundary value $(\hat{\varphi} \hat{f})\cdot (\hat{\phi} \hat{g})$ on $iI$.  
\end{proof}

\section{Sharp Mertens relation and the PNT}\label{section sharp Mertens}
We now prove Theorem \ref{thmainsimplified}. As anticipated in the Introduction, we will actually show a more general result. 

Our considerations apply to non necessarily discrete number systems. In a broader sense \cite{beurling,diamond-zhang}, a Beurling generalized number system is merely a pair of non-decreasing right continuous functions $N$ and $\Pi$ with $N(1)=1$ and $\Pi(1)=0$, both having support in $[1,\infty)$, and linked via the zeta function relation
\begin{equation}
\label{defzetaextended}
\zeta(s) :=\int^{\infty}_{1^{-}} x^{-s}\mathrm{d}N(x)= \exp\left(\int^{\infty}_{1}x^{-s}\mathrm{d}\Pi(x)\right),
\end{equation}
with convergence of both integrals on some half-plane. This implies that both $N$ and $\Pi$ have at most polynomial growth, and they determine one another in a unique fashion. Note that then (\ref{defzetaextended}) becomes equivalent to $\mathrm{d}N=\exp^{\ast_{M}}(\mathrm{d}\Pi)$, where the exponential is taken with respect to the multiplicative convolution of measures \cite{diamond1} (see also \cite[Chap.~2 and Chap.~3]{bateman-diamondbook} and \cite{diamond-zhangbook}). 

The Chebyshev function $\psi$ and the function $\psi_1$ involved in the sharp Mertens relation still make sense in this context if we employ the Stieltjes integral definitions in (\ref{ibpneq1}) and (\ref{ibpneq3}). As in classical number theory, the Chebyshev function has Mellin-Stieltjes transform
\begin{equation}
\label{treq2}
\int^{\infty}_{1} x^{-s} \mathrm{d}\psi(x) = -\frac{\zeta'(s)}{\zeta(s)}.
\end{equation}

 Our goal in this section is to provide a proof of the ensuing theorem. Our conditions for the equivalence between the PNT and the sharp Mertens relation are in terms of convolution averages of the remainder function $E$ in
\begin{equation}
\label{approximationN1}
N(x)=ax + xE(\log x), \quad x\geq 1,
\end{equation}
where we set $E(y)=0$ for $y<0$. In the sequel $\ast$ always denotes additive convolution.

\begin{theorem} \label{thmain} Suppose that a generalized number system satisfies the PNT. 
Let $K_1$ and $K_2$ be two kernels such that $\hat{K}_{j}(0)\neq0 $ and $\int_{-\infty}^{\infty}(1+|y|)^{1+\varepsilon}|K_{j}(y)|<\infty$ for some $\varepsilon>0$. If the remainder function $E$ determined by (\ref{approximationN1}), where $a>0$, satisfies
\begin{equation}
\label{thmainas1}
(E\ast K_1)(y)=o\left(\frac{1}{y}\right), \quad y\to\infty,
\end{equation}
and
\begin{equation}
\label{thmainas3} 
E \ast K_2 \in L^{1}(\mathbb{R}),            
\end{equation}
then the sharp Mertens relation 
\begin{equation}
\label{sharpMertens}
\psi_{1}(x) = \int^{x}_{1} \frac{\mathrm{d}\psi(t)}{t} = \log x + c + o(1)
\end{equation}
holds as well, with constant $c =-b/a - 1$ where $b = (\hat{K}_2(0))^{-1} \int_{-\infty}^{\infty} (E \ast K_2)(y)\mathrm{d}y$.                   
\end{theorem}
 
Implicitly in Theorem \ref{thmain} we need the existence of the convolutions $E\ast K_{j}$. This is always ensured by the PNT, as follows from the next simple proposition which delivers the bound $E(y)\ll_{\varepsilon}y^{1+\varepsilon}$ for every $\varepsilon>0$.
\begin{proposition}
\label{propboundpnt} The PNT implies the bound $N(x)\ll_{\varepsilon} x\log^{1+\varepsilon}x$, for each $\varepsilon>0$. 
\end{proposition}
\begin{proof} By the PNT we have $\log \zeta(\sigma)\sim -\log(\sigma-1)$, and so $\zeta(\sigma)\ll (\sigma-1)^{-1-\varepsilon}$ as $\sigma\to1^{+}$. Using that $N$ is non-decreasing, 
\begin{align*}
N(x)&\leq x\int_{1^{-}}^{x}u^{-1}\mathrm{d}N(u)\leq x\int_{1^{-}}^{x}u^{-1}\exp\left(1-\frac{\log u}{\log x}\right)\mathrm{d}N(u)
\\
&
\leq e x\zeta\left(1+1/\log x\right)\ll x\log^{1+\varepsilon}x. 
\end{align*}
\end{proof}

It is very easy to verify that Theorem \ref{thmain} contains Theorem \ref{thmainsimplified}. In fact, (\ref{thmainas3simplified}) is the same as $E\in L^{1}(\mathbb{R})$, which always yields (\ref{thmainas3}) for any kernel $K_{2}\in L^{1}(\mathbb{R})$. Furthermore, (\ref{thmainas1simplified}) implies  (\ref{thmainas1}) for any kernel such that $\int_{-\infty}^{\infty}|K_{1}(y)|(1+|y|)\mathrm{d}y<\infty$.

Note that we do not necessarily require in Theorem \ref{thmain} that $N$ has a positive asymptotic density, that is, that
\begin{equation}
\label{densityeq}
N(x)\sim ax.
\end{equation}
 In fact, this condition plays basically no role for our arguments. On the other hand, either (\ref{thmainas1simplified}) or  (\ref{thmainas3simplified}) automatically implies (\ref{densityeq}), as one may deduce from elementary  arguments in the second case. In the general case, (\ref{densityeq}) also follows from (\ref{thmainas1simplified}) or  (\ref{thmainas3simplified}) if one of the $K_{j}$ is a Wiener kernel. The next proposition collects this assertion as well as the useful bound $N(x)\ll x$, which is actually crucial for our proof of Theorem \ref{thmain} and turns out to be implied by its assumptions.

\begin{proposition}\label{pnteprop1} Let $N$ satisfy $N(x)\ll x\log^{\alpha}x$ and let $K$ be such that $\hat{K}(0)\neq0$ and $\int_{-\infty}^{\infty} (1+|y|)^{\alpha}|K(y)|\mathrm{d}y<\infty$, where $\alpha\geq0$. If $E\ast K\in L^{\infty}(\mathbb{R})$ or $E\ast K\in L^{1}(\mathbb{R})$, then $N(x)\ll x$ (and hence $E\in L^{\infty}(\mathbb{R})$). If additionally $K$ is a Beurling-Wiener kernel, that is, $\hat{K}(t)\neq 0$ for all $t\in \mathbb{R}$, and $(E\ast K)(y)=o(1)$ or $E\ast K\in L^{1}(\mathbb{R})$, then (\ref{densityeq}) holds. (Here $a=0$ is allowed).
\end{proposition}
\begin{proof} The bound on $N$ ensures that $E(y)=O((|y|+1)^{\alpha})$. Convolving $K$ with a test function $\varphi\in \mathcal{S}(\mathbb{R})$ yields $(E\ast K\ast \varphi)(y)=o(1) $ if $E\ast K\in L^{1}(\mathbb{R})$. Thus, replacing $K$ by $K\ast\varphi$ with $\hat{\varphi}(t)\neq0$ for all $t\in \mathbb{R}$ if necessary, we may just deal with the cases $E\ast K\in L^{\infty}(\mathbb{R})$ and $(E\ast K)(y)=o(1)$. Assume $E\ast K\in L^{\infty}(\mathbb{R})$. Applying the analogue of the Wiener division theorem \cite[Thm. II.7.3, p. 81]{korevaarbook} for the weighted Beurling algebra \cite{beurling1938,korevaarbook} $L^{1}_{\omega}(\mathbb{R})$ with non-quasianalytic weight function $\omega(y)=(1+|y|)^{\alpha}$, we obtain that $\hat{E}\in PM_{loc}(I)$ for some open interval $0\in I$ on which $\hat{K}(t)$ does not vanish. The Laplace transform of $E$, $\zeta(s+1)/(s+1)-a/s$, thus has local pseudomeasure boundary behavior on $iI$. Multiplying by the smooth function $(s+1)$ preserves the local pseudomeasure boundary behavior, and substituting $s+1$ by $s$ gives that $\zeta(s)-a/(s-1)$ has local pseudomeasure boundary behavior on $1+iI$, and so does $\zeta(s)$ because $a/(s-1)$ is actually a global pseudomeasure on $\Re e\:s=1$. Hence, $S(x)=N(e^{x})\ll e^{x}$ is a consequence of Lemma \ref{thtauberian3}. If we additionally know that $\hat{K}(t)$ never vanishes and $(E\ast K)(y)=o(1)$, the division theorem yields $\hat{E}\in PF_{loc}(\mathbb{R})$, or equivalently, that $\zeta(s)-a/(s-1)$ has local pseudofunction boundary behavior on the whole line $\Re e\:s=1$. The conclusion $N(x)\sim ax$ now follows from Theorem \ref{thtauberian2} applied to $S(x)=N(e^{x})$.
\end{proof}

Examples of kernels that can be used in Theorem \ref{thmain}, and for which one can apply Proposition \ref{pnteprop1} to deduce  (\ref{densityeq}) as a consequence of either  (\ref{thmainas1}) or (\ref{thmainas3}), are familiar summability kernels such as the Ces\`{a}ro-Riesz kernels $K(y)=e^{-y}(1-e^{-y})_{+}^{\beta}$ with $\beta\geq0$, the kernel of Abel summability $K(y)=e^{-y}e^{-e^{-y}}$, and the Lambert summability kernel $K(y)=e^{-y}p(e^{-y})$ with $p(u)=(u/(1-e^{u}))'$; see \cite{korevaarbook} and \cite[Sect.~1.5, p.~15]{ganeliusbook} for many other possible examples. Also, note that if we write $k_{1}(x)=x^{-1}K_{1}(-\log x)$, then (\ref{thmainas1}) takes the form
$$
\int_{1}^{\infty}\frac{N(u)}{u} k_{1}\left(\frac{u}{x}\right)\mathrm{d}u=ax\int_{0}^{\infty}k_{1}(u)\mathrm{d}u+o\left(\frac{x}{\log x}\right), \quad x\to\infty.
$$
In the previous examples we have $k_{1}(u)=(1-u)_{+}^{\beta}$, $k_{1}(u)=e^{-u}$, and $k_1(u)=p(u)$.

We now concentrate in showing Theorem \ref{thmain}. Our proof is based on the application of the Tauberian theorems from the previous section and two lemmas. We begin by translating the PNT and the sharp Mertens relation into boundary properties of the zeta function. 
\begin{lemma}
\label{lemmatranslationzeta}
A generalized number system satisfies the PNT if and only if
\begin{equation}
\label{treq3}
-\frac{\zeta'(s)}{\zeta(s)} - \frac{1}{s-1} 
\end{equation}   
admits local pseudofunction boundary behavior on $\Re e \: s = 1$;
the sharp Mertens relation (\ref{sharpMertens}) holds if and only if 
 \begin{equation}
\label{treq5}
-\frac{\zeta'(s)}{(s-1)\zeta(s) } - \frac{1}{(s-1)^{2}} -\frac{c}{s-1}
\end{equation}
has local pseudofunction boundary behavior on $\Re e \: s = 1$.
\end{lemma}
\begin{proof} 
The PNT is $S(x)=\psi(e^{x})\sim e^{x}$, which is equivalent to the local pseudofunction boundary behavior of (\ref{treq3}) by Theorem \ref{thtauberian2}  and (\ref{treq2}). Next, we apply Theorem \ref{thtauberian1} to the non-decreasing function $\psi_{1}(e^{y})$, which has Laplace transform $\mathcal{L}\{\psi_{1}(e^{y});s\} = -\zeta'(s+1)/(\zeta(s+1)s$; the sharp Mertens relation,
 \[
  \psi_{1}(e^{y}) =  \int^{e^{y}}_{1} \frac{\mathrm{d}\psi(u)}{u} = y + c + o(1),
\]
holds if and only if
\[
 -\frac{\zeta'(s+1)}{\zeta(s+1) s} - \frac{1}{s^{2}} - \frac{c}{s}
\]           
has local pseudofunction behavior on $\Re e \: s = 0$. 
\end{proof}

By multiplying (\ref{treq5}) by $(s-1)$, it follows that the sharp Mertens relation always implies the PNT. This should not be so surprising because the PNT can also be easily deduced from (\ref{sharpMertens}) via integration by parts, as was done in \cite[Prop. 2.1]{diamond-zhang}. The non-trivial problem is of course the converse implication of the conditional equivalence. Observe that $(s-1)^{-1}$ is smooth off $s=1$, so that $(it)^{-1}$ is a multiplier for $PF_{loc}(\mathbb{R}\setminus\{0\})$; consequently, the local pseudofunction boundary behavior of the two functions (\ref{treq3}) and (\ref{treq5}) becomes equivalent except at the boundary point $s=1$.

Summarizing, since we are assuming the PNT, our task reduces to show that (\ref{treq5}) has local pseudofunction boundary behavior on a boundary neighborhood of the point $s=1$. The next lemma allows us to extract some important boundary behavior information on the zeta function from the condition (\ref{thmainas3}). It also provides the alternative formula 
$$
c=-1-\frac{1}{a}\int^{\infty}_{1} \frac{N(x) - ax}{x^{2}} \mathrm{d}x= -\frac{1}{a}\lim_{x\to\infty} \left(\int_{1^{-}}^{x}t^{-1}\mathrm{d}N(t)-a \log x\right).
$$
for the constant in the sharp Mertens relation (\ref{sharpMertens}) if we additionally assume that $\hat{K}_{2}$ never vanishes in Theorem \ref{thmain}.

\begin{lemma} \label{lemintconv} Suppose that $N(x)\ll x\log x$ and $E\ast K \in L^{1}(\mathbb{R})$ where $\int_{-\infty}^{\infty}(1+|y|)K(y)\mathrm{d}y<\infty$ and $\hat{K}(0)\neq 0 $. (Here we may allow $a=0$.)  Then,
\begin{equation} \label{eqlemma2}
 \frac{1}{s-1} \left( \zeta(s) - \frac{a}{s-1}\right) - \frac{a+b}{s-1}
\end{equation}
has local pseudofunction behavior near $s =1$, where $b=(\hat{K}(0))^{-1} \int_{-\infty}^{\infty} (E \ast K)(y)\mathrm{d}y$.
 If in addition $\hat{K}(t)\neq 0$ for all $t$, then
the integral
\begin{equation} 
	\label{thmainas2} 
	 \int^{\infty}_{1} \frac{N(x) - ax}{x^{2}}\: \mathrm{d}x = \int^{\infty}_{0}E(y)\mathrm{d}y
\end{equation}
converges to $b$ and (\ref{eqlemma2}) has local pseudofunction boundary behavior on the whole line $\Re e \: s = 1$.
 
\end{lemma}
\begin{proof}
Set $\tau(x) := \int^{x}_{0} E(y)\mathrm{d}y$. Proposition \ref{pnteprop1} yields $N(x)\ll x$, i.e., $E(y)=O(1)$.  
Now, for $x>0$, $(\tau\ast K)(x)=\int_{-\infty}^{x}(E\ast K)(y)\mathrm{d}y=b\hat{K}(0)+o(1)=b\cdot (\chi_{[0,\infty)}\ast K)(x)+o(1)$, because the latter integral is even absolutely convergent. By applying the division theorem for the Beurling algebra  \cite{beurling1938,korevaarbook} with weight function $(1+|y|)$, we obtain that the Fourier transform of $\tau(x)-b\chi_{[0,\infty)}(x)$ is a pseudofunction in a neighborhood of the origin. So,
\begin{equation} \label{eqlemma1}
 \mathcal{L}\{\tau;s\}-\frac{b}{s} = \frac{1}{s} \left( \frac{\zeta(s+1)}{s+1} - \frac{a}{s}\right) - \frac{b}{s}
\end{equation}
has local pseudofunction boundary behavior near $s = 0$. Multiplication of (\ref{eqlemma1}) by $(s+1)$ produces an equivalent expression in terms of local pseudofunction boundary behavior on a boundary neighborhood of  $s=0$. Thus, since the term $-b$ is negligible, we obtain that 
\begin{equation}
\label{eqlemma3}
 \frac{1}{s} \left( \zeta(s+1) - \frac{a}{s}\right) - \frac{a+b}{s}
\end{equation}
has local pseudofunction boundary behavior near $s=0$. Under the extra assumption that $K$ is a Beurling-Wiener kernel, we obtain that the Fourier transform of $\tau(x)-b\chi_{[0,\infty)}(x)$ belongs to $PF_{loc}(\mathbb{R})$ again from the division theorem. Thus (\ref{eqlemma1}) and (\ref{eqlemma3}) both have local pseudofunction boundary behavior on the whole line $\Re e\:s = 0$. Noticing that $\tau$ is slowly decreasing, we can apply Theorem \ref{thtauberian1} to conclude $\tau(x)=b+o(1)$. 
\end{proof}

We can now give a proof of Theorem \ref{thmain}.

\begin{proof}[Proof of Theorem \ref{thmain}] As we have already mentioned, it suffices to establish the local pseudofunction boundary behavior of (\ref{treq5}) near $s=1$.
We will show that 
\begin{equation}
\label{preq1}
 \frac{\zeta'(s)}{\zeta(s)}\left(\zeta(s)  - \frac{a}{s-1}\right) + \frac{a+b}{s-1}
\end{equation}    
and
\begin{equation}
\label{preq2}
 \zeta'(s) + \frac{a}{(s-1)^{2}}
\end{equation}    
both admit local pseudofunction boundary behavior (near $s=1$).  This would prove the theorem. Indeed, subtracting (\ref{preq2}) from (\ref{preq1}) gives that
\begin{equation}
\label{preq3}
   -\frac{a\zeta'(s)}{\zeta(s)(s-1)} + \frac{a+b}{s-1}- \frac{a}{(s-1)^{2}}
\end{equation} 
has local pseudofunction boundary behavior. This yields the local pseudofunction boundary behavior of (\ref{treq5}) with $c=-1-b/a$ after division of (\ref{preq3}) by $a$, which implies the sharp Mertens relation in view of Lemma \ref{lemmatranslationzeta}.

For the local pseudofunction boundary behavior of (\ref{preq1}), it is enough to show that
\begin{equation}
\label{preq4}
   \left( - \frac{\zeta'(s)}{\zeta(s)} - \frac{1}{s-1}\right) . \left(\zeta(s) - \frac{a}{s-1}\right)
\end{equation}                                                                     
has local pseudofunction boundary behavior near $s=1$. In fact, (\ref{preq1}) is  (\ref{preq4}) minus (\ref{eqlemma2}), and our claim then follows from Lemma \ref{lemintconv}.
Now, applying the Wiener division theorem \cite[Thm.~7.3]{korevaarbook}, the fact $E\in L^{\infty}(\mathbb{R})$ (Proposition \ref{pnteprop1}), and the hypothesis $E\ast K_{2}\in L^{1}(\mathbb{R})$, we obtain that $\hat{E}\in A_{loc}(I)$ for some neighborhood $I$ of $t=0$. Thus, the second factor $\zeta(s) - a/(s-1)$ of (\ref{preq4})  has $A_{loc}$-boundary behavior on $1+iI$. Thus, the PNT and Lemma \ref{productlpfbb} give the local pseudofunction boundary behavior of (\ref{preq4}).

Finally, it remains to verify that (\ref{preq2}) has local pseudofunction boundary behavior near $s=1$. Note that
$$
\zeta'(s) + \frac{a}{(s-1)^{2}}= \mathcal{L}\{E,s-1\}+s\frac{d}{ds}(\mathcal{L}\{E,s-1\}).
$$ 
Therefore, we should show that the derivative of $\hat{E}(t)$ is a pseudofunction in a neighborhood of $t=0$. Let $0\in I$ be an open interval such that $\hat{E}\in PF_{loc}(I)$ and $\hat{K}_{1}(t)\neq0$ for all $t\in I$. If $\phi=\hat{\varphi}\in\mathcal{D}(I)$, we can write $\varphi=K_{1}\ast Q$, where 
\begin{equation}
\label{diveq1}
\int_{-\infty}^{\infty}(1+|y|)|Q(y)|\mathrm{d}y<\infty,
\end{equation} 
as follows from the division theorem for non-quasianalytic Beurling algebras \cite{beurling1938,korevaarbook}. Thus, making use of (\ref{thmainas1}) and (\ref{diveq1}), we have $(\varphi\ast E)(h)= \int_{-\infty}^{\infty} (E\ast K_{1})(y+h)Q(-y)\mathrm{d}y=o(1/h)$ as $h\to\infty$. But since $\operatorname*{supp} E\subseteq [0,\infty)$, we also have $(\varphi\ast E)(h)=o(1/|h|)$ as $-h\to\infty$. In conclusion, $\langle \hat{E}(t),e^{-ith}\phi(t)\rangle=o(1/|h|)$ for any $\phi\in\mathcal{D}(I)$. Now, if $\phi\in\mathcal{D}(I)$, 
\begin{align*}
\langle \hat{E}'(t),e^{-ith}\phi(t)\rangle&=-\langle \hat{E}(t),e^{-ith}\phi'(t)\rangle+ih\langle \hat{E}(t),e^{-ith}\phi(t)\rangle
\\
&= o(1/|h|)+o(1)=o(1).
\end{align*}
So, $\hat{E}'\in PF_{loc}(I)$ and the proof is complete.
\end{proof}

The same method as above leads to the following corollary.

\begin{corollary}
\label{sharpMertensCor} The sharp Mertens relation and (\ref{thmainas3}), for some kernel with $\hat{K}_{2}(0)\neq0$ and $\int_{-\infty}^{\infty} (1+|y|)|K_{2}(y)|\mathrm{d}y<\infty$, imply (\ref{thmainas1}) for any $K_{1}$ such that $\int_{-\infty}^{\infty} (1+|y|)|K_{1}(y)|\mathrm{d}y<\infty$ and $\operatorname*{supp} \hat{K}_{1}$ is a compact subset of the interior of $\operatorname*{supp} \hat{K}_{2}$. (Here we may have $a=0$.)
\end{corollary}
\begin{proof} Let us first verify that the convolution $E\ast K_{2}$ is well-defined. We show that the sharp Mertens relation gives the bound $N(x)=o( x\log x)$. In fact, integrating by parts, we see that $\psi_{1}(x)=\log x+c+o(1)$ implies
$$
\int_{1}^{x}\frac{\mathrm{d}\Pi(u)}{u}=\int_{1}^{x}\frac{\mathrm{d}\psi_{1}(u)}{\log u}= \log \log x+c_{1}+O\left(\frac{1}{\log x}\right),
$$
for some $c_1$. This readily yields $\log \zeta(\sigma)= -\log(\sigma-1)+c_{1}-\gamma+o(1)$ as $\sigma\to1^{+}$, so that
$$
\zeta(\sigma)\sim \frac{e^{c_1-\gamma}}{\sigma-1}. 
$$
Applying the Karamata Tauberian theorem \cite{korevaarbook}, we conclude
$$
\int_{1^{-}}^{x}\frac{\mathrm{d}N(u)}{u}\sim e^{c_1-\gamma}\log x,
$$
and hence integration by parts yields $N(x)=o( x\log x)$ as claimed. Proposition \ref{pnteprop1} then allows us to improve the estimate to $N(x)\ll x$.

Set $U=\{t\in\mathbb{R}:\: \hat{K}_{2}(t)\neq 0\}$. Next, the proof of Lemma \ref{lemintconv}, a Wiener division argument, and Lemma \ref{productlpfbb} give at once that (\ref{eqlemma2}) and (\ref{preq4}), and hence (\ref{preq1}), all have local pseudofunction boundary behavior on $1+iU$. Using Lemma \ref{lemmatranslationzeta}, multiplying (\ref{treq5}) by $a$, and subtracting the resulting expression from  (\ref{preq1}), we obtain that ($d=ca-a+b$)
$$
\zeta'(s)+\frac{a}{(s-1)^{2}}-\frac{d}{(s-1)}=\mathcal{L}\left\{E;s-1\right\}+s\frac{d}{ds}(\mathcal{L}\{E,s-1\})-\frac{d}{(s-1)}.
$$
also has local pseudofunction boundary behavior on $1+iU$. Since $\hat{E}\in A_{loc}(U)\subset PF_{loc}(U)$, we must have $i\hat{E}'(t)-d(it+0)^{-1}\in PF_{loc}(U)$, where $(it+0)^{-1}$ denotes the distributional boundary value of $1/s$ on $\Re e\: s=0$. If we now take an arbitrary $\varphi\in\mathcal{S}(\mathbb{R})$ such that $\hat{\varphi}\in \mathcal{D}(U)$, we obtain that
\begin{align*}
y(E\ast\varphi)(y)&=d\hat{\varphi}(0)+o(1)+
\int_{-\infty}^{y} ((y-u)E(y-u)-d)\varphi(u)\mathrm{d}u
\\
&
\quad +\int_{-\infty}^{y} E(y-u)u\varphi(u)\mathrm{d}u
\\
&
=d\hat{\varphi}(0)+o(1), \quad y\to\infty,
\end{align*}
namely,
$$
(E\ast\varphi)(y)\sim \frac{d}{y}\hat{\varphi}(0) ,\quad y\to\infty.
$$
Since $0\in U$ and $(E\ast \varphi)\in L^{1}(\mathbb{R})$ because $\hat{E}\in A_{loc}(U)$, we conclude $d=0$, upon taking $\varphi$ with $\hat{\varphi}(0)\neq0$. So, $(E\ast\varphi)(y)=o(1/y)$ as $y\to\infty$ for any $\varphi$ with $\hat{\varphi}\in \mathcal{D}(U)$. If $K_1$ satisfies the conditions from the statement, we can write $K_{1}=K_{1}\ast\varphi$ for any $\hat{\varphi}\in\mathcal{D}(U)$ being equal to 1 on $\operatorname*{supp}\hat{K}_{1}$. Therefore, $(E\ast K_{1})(y)= ((E\ast\varphi)\ast K_{1})(y)=o(1/y)$ as well.
\end{proof}

\begin{remark}\label{pnter2} We end this section with some remarks on possible variants for the assumptions on the convolution kernels in Theorem \ref{thmain}.
\begin{itemize}
\item[(i)] If one the of the kernels $K_{1}$ or $K_{2}$ is non-negative, we can replace the requirements 
\begin{equation}
\label{eq1rk}
\int_{-\infty}^{\infty}(1+|y|)^{1+\varepsilon}|K_{j}(y)|\mathrm{d}y<\infty
\end{equation}
by the weaker ones
\begin{equation*}
\int_{-\infty}^{\infty}(1+|y|)|K_{j}(y)|\mathrm{d}y<\infty.
\end{equation*}
In fact, the conditions (\ref{eq1rk}) were only used in order to ensure the existence of the convolutions via Proposition \ref{propboundpnt}, but the rest of our arguments still works if we would have a priori known the better bound $N(x)\ll x\log x$. Now, if $K$ is non-negative, a bound $(E\ast K)(y)=O(1)$ necessarily implies $N(x)\ll x$. In fact, set $T(y)=e^{-y}N(e^{y})$. Since $N$ is non-decreasing, we have $T(h)e^{-y}\leq T(y+h)$ for $h\geq 0$. Setting $C^{-1}=\int_{-\infty}^{0} e^{y}K(y)\mathrm{d}y$, we obtain
$$
T(h)\leq C\int_{0}^{\infty} T(y+h) K(-y)\mathrm{d}y\leq C (T\ast K)(h)\ll 1.
$$
\item[(ii)] The Wiener type division arguments can be completely avoided if $\hat{K}_{j}$ is $C^{\infty}$ near 0 (or on $\mathbb{R}$ whenever the global non-vanishing is required). Indeed, in this case the division can be performed in the Fourier transform side as the multiplication of a distribution by a smooth function, a trivial procedure in distribution theory. In particular, Theorem \ref{thmainsimplified} can be shown without appealing to Wiener type division theorems.
\item [(iii)] In connection with the previous comment, one can even drop the integrability conditions on $K_{j}$ and employ distribution kernels. It is well known that the space of convolutors for tempered distributions $\mathcal{O}'_{C}(\mathbb{R})$ satisfies $\mathcal{F}(\mathcal{O}'_{C}(\mathbb{R}))\subset C^{\infty}(\mathbb{R})$ \cite{estrada-kanwal,p-s-v,schwartz}. So, Theorem \ref{thmain} holds if we assume that $K_1,K_2\in \mathcal{O}'_{C}(\mathbb{R})$ are such that $\hat{K}_{j}(0)\neq 0$ and (\ref{thmainas1}) and (\ref{thmainas2}) are satisfied. A list of useful kernels belonging to $\mathcal{O}'_{C}(\mathbb{R})$ and having nowhere vanishing Fourier transforms is discussed in Ganelius' book \cite[Sect.~1.5, p.~15]{ganeliusbook}. An even more general result is possible: we can weaken $K_1,K_2\in \mathcal{O}'_{C}(\mathbb{R})$ to $K_{1},K_{2}\in \mathcal{S}'(\mathbb{R})$, $\hat{K}_{j}$ is $C^{\infty}$ and non-zero in a neighborhood of $0$, and $E\ast K_{j}$ exists in the sense of $\mathcal{S}'$-convolvability \cite{kaminski82}. 
\end{itemize}
\end{remark}

\section{The Landau relations $M(x)=o(x)$ and $m(x)=o(1)$} \label{sectionotherequivalences}
The section is devoted to conditions on $N$ that imply the equivalence between (\ref{eqlandaurelation}) and (\ref{eqlandausharprelation}). As in the previous section, we do not need to assume that the generalized number system is discrete. In the general case, we define $M(x)=\int_{1^{-}}^{x}\mathrm{d}M(u)$, where the measure $\mathrm{d}M$ is the multiplicative convolution inverse \cite{diamond1,diamond-zhangbook} of $\mathrm{d}N$, that is, $\mathrm{d}M=\exp^{\ast_{M}}(-\mathrm{d}\Pi)$. So, we consider the Landau relations 
\begin{equation}
\label{eqMoebius1}
M(x)=o(x)
\end{equation}
and 
\begin{equation}
\label{eqMoebius2}
m(x)=\int_{1^{-}}^{x}\frac{\mathrm{d}M(u)}{u}=o(1).
\end{equation}

It is a simple fact that (\ref{eqMoebius2}) always implies (\ref{eqMoebius1}). This follows from integration by parts:
$$
M(x)=\int_{1^{-}}^{x}u\mathrm{d}m(u)= xm(x)- \int_{1}^{x}m(u)\mathrm{d}u=o(x).
$$

Our aim in this section is to show that the conditional converse implication holds under a weaker hypothesis than in Theorem \ref{thmainmobius}. 
\begin{theorem} \label{thmmobius} Suppose that $N(x)\ll x$ and there are $a > 0$ and $K \in L^{1}(\mathbb{R})$ such that $\hat{K}(0) \neq 0$ and
	\begin{equation} \label{eqmainconditionmobius}
	E \ast K \in L^{1}(\mathbb{R}),             
	\end{equation}
where $E$ is the remainder function determined by (\ref{approximationN1}). Then, 
\[
M(x)=o(x) \quad \mbox{implies} \quad m(x)=o(1).
\]
\end{theorem}
\begin{proof}
Note first that 
$$ 
m(e^{x})=\frac{M(e^{x})}{e^{x}}+ \int_{0}^{x}\frac{M(e^{y})}{e^{y}}\mathrm{d}y
$$
is slowly decreasing because $M(e^{x})=o(e^{x})$. Since $\mathcal{L}\{\mathrm{d}M(e^{x});s\} = 1/\zeta(s)$, the Abelian part of Theorem \ref{thtauberian2} gives that $M(x)=o(x)$ implies that $1/\zeta(s)$ has local pseudofunction boundary behavior on the whole line $\Re e\:s=1$. On the other hand, the Laplace transform of the slowly decreasing function $m(e^{x})$ is $\mathcal{L}\{m(e^{x});s\} = 1/(s\zeta(s+1))$, $\Re e\: s>0 $; one then deduces from Theorem \ref{thtauberian1} that it suffices to show that this analytic function admits local pseudofunction behavior on the line $\Re e \: s = 0$. Since $1/s$ is smooth away from $s =0$, it is enough to establish the local pseudofunction boundary behavior of $1/(s\zeta(s+1))$ near $s = 0$. We now verify the latter property. By employing the Wiener division theorem \cite[Thm. 7.3]{korevaarbook}, we have that $\hat{E}\in A_{loc}(I)$ for some open interval $I$ containing $0$. This leads to the $A_{loc}$-boundary behavior of $(\zeta(s+1)/(s+1) - a/s)$ on $iI$; multiplying by $(s+1)$ and adding $a$, we conclude that $\zeta(s+1)-a/s$ has boundary values in the local Wiener algebra on the boundary line segment $iI$.  
The local pseudofunction boundary behavior of $1/s\zeta(s+1)$  near $s = 0$ now follows from
\[
 \frac{1}{s\zeta(s+1)} = -\frac{1}{a\zeta(s+1)} \cdot \left(\zeta(s+1) - \frac{a}{s}\right)+\frac{1}{a}
\]
and Lemma \ref{productlpfbb}.
\end{proof}

\begin{remark}
\label{pnterk3} 
We mention some variants of Theorem \ref{thmmobius}:
\begin{itemize}
\item [(i)] If the kernel $K$ is non-negative, the bound $N(x)\ll x$ becomes superfluous because it is implied by (\ref{eqmainconditionmobius}), see Remark \ref{pnter2}(i).
\item [(ii)] If only a bound $N(x)\ll x\log^{\alpha} x$,  $\alpha>0$, is initially known for $N$, we can compensate it by strengthening the assumption on $K$ to $\int_{-\infty}^{\infty}(1+|y|)^{\alpha}|K(y)|\mathrm{d}y<\infty$. As a matter of fact, the bound $N(x)\ll x$ would then be implied by Proposition~\ref{pnteprop1}.
\item [(iii)] The comments from Remark \ref{pnter2}(ii) and Remark \ref{pnter2}(iii) also apply to Theorem \ref{thmmobius}. In order to use distribution kernels $K$, one assumes that $N(x)\ll x\log^{\alpha} x$ for some $\alpha$ to ensure that $E$ is a tempered distribution. 

\end{itemize}
\end{remark}

\section{Examples} \label{pnte examples}
We shall now construct examples in order to prove  Proposition \ref{prop1pnteex}. In addition, we give an example of a generalized number system such that $M(x)=o(x)$, $m(x)=o(1)$, the condition (\ref{eqmainconditionmobius}) holds for any kernel $K\in\mathcal{S}(\mathbb{R})$, but for which (\ref{thmainas3simplified}) is not satisfied.

\begin{example} (Proposition \ref{prop1pnteex}(i)). \label{pnteex1} By constructing an example \cite{diamond-zhang,diamond-zhangbook}, Diamond and Zhang showed that 
the PNT and 
\begin{equation}
\label{Odensitycond1}
N(x)=ax+O\left(\frac{x}{\log x}\right)
\end{equation}
do not imply the sharp Mertens relation in general. We slightly modify their arguments to produce an example that satisfies the stronger relation (\ref{thmainas1simplified}).

 Let $\omega$ be a positive non-increasing function on $[1,\infty)$ such that 
\begin{equation}
\label{pnteex1eq1}
\int_{2}^{\infty}\frac{\omega(x)}{x\log x}\mathrm{d}x=\infty,
\end{equation}
and 
\begin{equation}
\label{pnteex1eq1.1} \frac{\omega(x^{1/n})}{\omega(x)}\leq Cn^{\alpha},
\end{equation}
where $C,\alpha>0$. For example, $\omega(x)=1/\log\log x$ for $x\geq e^{e}$ and $\omega(x)=1$ for $x\in [1,e^{e}]$ satisfies (\ref{pnteex1eq1}) and the better inequality $\omega(x^{1/n})/\omega(x)\leq 1+\log n$. 

We construct here  a generalized number system satisfying the PNT, the asymptotic estimate
\begin{equation}
\label{pnteex1eq2}
N(x)=ax+O\left(\frac{x\omega(x)}{\log x}\right),
\end{equation}
for some $a>0$, but for which the sharp Mertens relation fails. Upon additionally choosing $\omega$ with $\omega(x)=o(1)$, we obtain (\ref{thmainas1simplified}).

We prove that 
$$
\mathrm{d}\Pi(u)= \frac{1-u^{-1}}{\log u}\mathrm{d}u+  \left(\frac{1-u^{-1}}{\log u}\right)^{2}\omega(u)\mathrm{d}u
$$
fulfills our requirements. ($\omega(u)=1$ is the example from \cite{diamond-zhang,diamond-zhangbook}). For the PNT,
\begin{align*}
\Pi(x)&=\operatorname*{Li}(x)+ O(\log\log x)+\left(\int_{2}^{\sqrt{x}}+\int_{\sqrt{x}}^{x}\right)\frac{O(1)}{\log^{2} u}\mathrm{d}u
\\
&
=\frac{x}{\log x}+O\left(\frac{x}{\log^{2} x}\right),
\end{align*}
because $\omega$ is bounded. Next, by (\ref{pnteex1eq1}),
$$
\psi_{1}(x)-\log x=\int_{1}^{x}\frac{\log u}{u}\mathrm{d}\Pi(u)-\log x\geq -1+\frac{1}{4}\int_{2}^{x} \frac{\omega(u)}{u\log u}\mathrm{d}u\to\infty.
$$
To get (\ref{pnteex1eq2}), we literally apply the same convolution method as in \cite[Sect. 14.4]{diamond-zhangbook}. Using that $\exp^{\ast_{M}} ((1-u^{-1})/\log u \:\mathrm{d}u)= \chi_{[1,\infty)}(u)\mathrm{du}+\delta(u-1)$ (with $\delta$ the Dirac delta) and that the latter measure has distribution function $x$ (for $x\geq1$), one easily checks that
$$
N(x)=x\int_{1^{-}}^{x}\exp^{\ast_{M}}\left(\mathrm{d}\nu\right)=x\left(1+\sum_{n=1}^{\infty}\frac{1}{n!}\int_{1}^{x}\mathrm{d}\nu^{\ast_{M}n}\right),
$$
where 
$$
\mathrm{d}\nu(u)=\frac{(1-u^{-1})^{2}}{u\log^{2} u}\omega(u)\mathrm{d}u.
$$
Since 
$$
c=\int_{1}^{\infty}\mathrm{d}\nu\leq O(1)\int_{1}^{\infty}\frac{(1-u^{-1})^{2}}{u\log^{2} u}\mathrm{d}u<\infty,
$$
we obtain $N(x)=x(e^{c}-R(x))$ with
$$R(x)=\sum_{n=1}^{\infty}\frac{1}{n!}\int_{x}^{\infty}\mathrm{d}\nu^{\ast_{M}n}=\sum_{n=1}^{\infty}\frac{1}{n!}\idotsint_{x< u_{1}u_2\dots u_{n}}\mathrm{d}\nu(u_1)\dots \mathrm{d}\nu(u_{n}).$$

We estimate $R(x)$ for $x\geq2$. Since all variables in the multiple integrals from the above summands are greater than 1, introducing the constraint $u_n>x$ gives $\idotsint_{x< u_{1}u_2\dots u_{n}}\geq \int_{1}^{\infty}\dots\int_{1}^{\infty}\int_{u_n>x}$, namely,
$$
\int_{x}^{\infty}\mathrm{d}\nu^{\ast_{M}n}\geq \frac{c^{n-1}}{4}\int_{x}^{\infty}\frac{\omega(u)}{u\log^{2} u}\mathrm{d}u.
$$
On the other hand, as was noticed in \cite{diamond-zhang}, at least one of the variables $u_j$ should be $> x^{1/n}$, and therefore
$$
\int_{x}^{\infty}\mathrm{d}\nu^{\ast_{M}n}\leq nc^{n-1}\int_{x^{1/n}}^{\infty}\frac{\omega(u)}{u\log^{2} u}\mathrm{d}u\leq Cn^{2+\alpha}c^{n-1}\int_{x}^{\infty}\frac{\omega(u)}{u\log^{2} u}\mathrm{d}u.
$$
Adding up these estimates, we obtain
\begin{equation*}
%\label{pnteex1eq3}
0<  C_{1}\int_{x}^{\infty}\frac{\omega(u)}{u\log^{2} u}\mathrm{d}u\leq\frac{e^{c}x-N(x)}{x}\leq C_{2} \int_{x}^{\infty}\frac{\omega(u)}{u\log^{2} u}\mathrm{d}u
\end{equation*}
with 
$$C_{1}=\frac{e^{c}-1}{4c} \quad \mbox{and} \quad C_{2}=C\sum_{n=0}^{\infty}\frac{c^{n}}{n!}(n+1)^{\alpha+1}.$$
In particular, the upper bound yields (\ref{pnteex1eq2}) with $a=e^{c}$ because $\omega$ is non-increasing.
\end{example}

\smallskip

\begin{example}[Proposition \ref{prop1pnteex}(ii)]\label{pnteex2} We now give an example to prove that the PNT and the condition (\ref{thmainas3simplified}) do not imply a sharp Mertens relation in general. The generalized number system has the form
\begin{equation}\label{pnteex2eq1}
\mathrm{d}\Pi(u)=  \frac{1-u^{-1}}{\log u}\mathrm{d}u+  \frac{f(\log u)}{\log^{2} u}\chi_{[A,\infty)}(u)\mathrm{d}u
\end{equation}
where $f$ will be suitably chosen below, and $A\geq e$. We suppose that $|f(y)|\leq y/2$ on $[\log A,\infty)$ in order to ensure that $\mathrm{d}\Pi$ is a positive measure. Observe that we do not assume here that $f$ is non-negative, actually letting $f$ be oscillatory is important for our construction. We start with a preliminary lemma that gives a condition for the function $N$ to satisfy (\ref{thmainas3simplified}).

\begin{lemma}\label{lemma suff L1} Assume that $\int_{\log A}^{\infty}y^{-2}|f(y)|\mathrm{d}y<\infty$ and let $a=\exp(\int_{\log A}^{\infty}y^{-2}f(y)\mathrm{d}y)$. Then, (\ref{thmainas3simplified}) holds if 
\begin{equation}
\label{pnteex2eq2} \int_{x}^{\infty} \frac{f(y)}{y^{2}}\mathrm{d}y\in L^{1}(\mathbb{R}).
\end{equation}
Conversely, (\ref{thmainas3simplified}) implies (\ref{pnteex2eq2}) if $\int_{\log A}^{\infty}y^{-2}|f(y)|\mathrm{d}y<\pi$. 
\end{lemma}
\begin{proof} The method of this proof is essentially due to Kahane \cite[p. 633]{kahane2}. Denote as $B(\mathbb{R})$ the Banach algebra of Fourier transforms of finite Borel measures. Note that the elements of $B(\mathbb{R})$ are multipliers for the Wiener algebra $A(\mathbb{R})$. We have to show that $\hat{E}\in A(\mathbb{R})$ if (\ref{pnteex2eq2}) holds, where as usual $E(y)=e^{-y}N(e^{y})-a$. Write 
$$
L(t)=\int_{\log A}^{\infty} \frac{e^{-ity}f(y)}{y^{2}}\mathrm{d}y,
$$
and note that $a=e^{L(0)}$,
$$
\zeta(1+it)=\frac{(1+it)e^{L(t)}}{it},
$$
and
$$
\hat{E}(t)= \frac{\zeta(1+it)}{(1+it)}-\frac{e^{L(0)}}{it}= e^{L(0)} \frac{e^{L(t)-L(0)}-1}{L(t)-L(0)} \cdot \frac{L(t)-L(0)}{it}.
$$
We have that $L(t)-L(0)\in B(\mathbb{R})$, because it is the Fourier transform of the finite measure $y^{-2}f(y)\chi_{[\log A,\infty)}(y)\: \mathrm{d}y - L(0)\delta(y)$, where $\delta$ is the Dirac delta. Since entire functions act on Banach algebras, we have that $(e^{L(t)-L(0)}-1)/(L(t)-L(0))\in B(\mathbb{R})$ is a multiplier for $A(\mathbb{R})$. Therefore, $\hat{E}\in A(\mathbb{R})$ if $(L(t)-L(0))/(it)\in A(\mathbb{R})$. The rest follows by noticing that the latter function is the Fourier transform of $-\int_{x} ^{\infty}y^{-2}f(y)\mathrm{d}y$. Conversely, $(L(t)-L(0))/(e^{L(t)-L(0)}-1)\in B(\mathbb{R})$ because $z/(e^{z}-1)$ is analytic in the disc $|z|<2\pi$ and $\|L-L(0)\|_{B(\mathbb{R})}\leq 2\int_{\log A}^{\infty}y^{-2}|f(y)|\mathrm{d}y<2\pi$. 
\end{proof}

We now set out the construction of $f$. Select a non-negative test function $\varphi\in \mathcal{D}(-1/2,1/2)$ with $\varphi(0)=1$. In the sequel we consider the non-negative function
$$
g(x)=\sum_{n=1}^{\infty} \varphi(n^{3}(x-n-1/2)).
$$
It is clear that
$$
\int_{0}^{\infty}g(x)\mathrm{d}x=\left(\int_{-1/2}^{1/2}\varphi(x)\mathrm{d}x\right) \left(\sum_{n=1}^{\infty}\frac{1}{n^{3}}\right), \ \  |g'(x)|\ll x^{3}, \ \  \mbox{ and } \ \ \int_{x}^{\infty}g(y)\mathrm{d}y\ll \frac{1}{x^{2}}.
$$

We set $f(y)=g'(\log y)$ in (\ref{pnteex2eq2}) and choose $A$ so large that $|f(y)|\leq y/2$ for $y\geq \log A$. The PNT holds,
\begin{align*}
\Pi(x)&= \frac{x}{\log x}+O\left(\frac{x}{\log^{2} x}\right)+O\left(\int_{A}^{x} \frac{|g'(\log \log u)|}{\log^{2} u}\mathrm{d}u\right)
\\
&
=\frac{x}{\log x}+O\left(\frac{x(\log\log x)^{3}}{\log^{2} x}\right).
\end{align*}
Furthermore, since $\limsup_{x\to\infty} g(x)=1$, $\liminf_{x\to\infty}g(x)=0$, and
\begin{align*}
\psi_{1}(x)&=\log x-1+\int_{A}^{x}\frac{f(\log u)}{u\log u}\mathrm{d}u+o(1)
\\
&
=\log x-1+g(\log \log x)-g(\log \log A)+o(1),
\end{align*}
we obtain that the sharp Mertens relation does not hold. It remains to check (\ref{thmainas3simplified}) via Lemma \ref{lemma suff L1}, that is, we verify that (\ref{pnteex2eq2}) is satisfied. Indeed,
\begin{align*}
 \int_{x}^{\infty} \frac{f(y)}{y^{2}}\mathrm{d}y&= -\frac{g(\log x)}{x}+\int_{x}^{\infty}\frac{g(\log y)}{y^{2}}\mathrm{d}y
 \\
 &
 = -\frac{g(\log x)}{x}+O\left( \frac{1}{x\log^{2} x}\right)\in L^{1}(\mathbb{R}).
 \end{align*}
\end{example}

\smallskip

\begin{example}\label{pnteex3} We now provide an example that shows that there are situations in which (\ref{thmainas3simplified}) fails, but Theorem \ref{thmmobius} could still apply to deduce the equivalence between the Landau relations. In addition, this example satisfies
\begin{equation}
\label{pnteeqex4}
N(x)=ax+\Omega_{\pm}\left(\frac{x}{\log^{1/2} x}\right).
\end{equation}

Consider 
\[
\mathrm{d}\Pi(u)=\frac{1+\cos(\log u)}{\log u} \chi_{[2,\infty)}(u)\mathrm{d}u.
\]
This continuous generalized number system is a modification of the one used by Beurling to show the sharpness of his PNT \cite{beurling}. Note the PNT fails for $\Pi$, one has instead
$$ 
\Pi(x)=\frac{x}{\log x}\left(1+\frac{\sqrt{2}}{2}\cos \left(\log x-\frac{\pi}{4}\right)\right)+O\left(\frac{x}{\log^{2} x}\right).
 $$

We have
$$
\log\zeta(s)=-\log (s-1)-\frac{1}{2}\log (s-1-i)-\frac{1}{2}\log (s-1+i)+G(s),
$$
so that,
$$
\zeta(s)= \frac{e^{G(s)}}{(s-1)\sqrt{1+(s-1)^{2}}}\:,
$$
where $G(s)$ is an entire function. 
If we set $a=e^{G(1)}$, we obtain that
$ \zeta(s)-a/(s-1)
$
has $L^{1}_{loc}$-boundary behavior on $\Re e\:s=1$, so that, by Theorem \ref{thtauberian2},
$$
N(x)\sim ax,
$$
although $\zeta(1+it)$ is unbounded at $t=\pm i$. The condition (\ref{thmainas3simplified}) would imply continuity of $\zeta(1+it)$ at all $t\neq0$, hence, we must have
\begin{equation}
\label{pnteqex4.0}
\int_{1}^{\infty}\left|\frac{N(x)-ax}{x^{2}}\right| \mathrm{d}x=\infty.
\end{equation}
Using the same method as in \cite[Sect.~5]{d-s-v}, one can even show that there are constants $d_{0},d_1,\dots$ and $\theta_{0},\theta_2,\dots$ with $d_0\neq 0$ such that  

\begin{align}
\label{pnteeqex4.1}
N(x)&\sim ax+ \frac{x}{\log^{1/2}x}\sum^{\infty}_{j=0}d_{j}\frac{\cos(\log x+\theta_{j})}{\log^{j}x}
\\
\nonumber
&= ax+d_{0}\frac{x\cos(\log x+ \theta_{0})}{\log^{1/2} x}+ O\left(\frac{x}{\log^{3/2}x}\right).
\end{align}
This yields (\ref{pnteeqex4}), and also another proof of (\ref{pnteqex4.0}).

On the other hand, $\zeta(s)-a/(s-1)$ has an analytic extension to $1+i\mathbb{R}\setminus\{1\pm i\}$; in particular, it has  $A_{loc}$-boundary behavior on  $1+i(-1,1)$. Thus, (\ref{eqmainconditionmobius}) holds for any kernel $K\in L^{1}(\mathbb{R})$ with $\operatorname*{supp} \hat{K}\subset (-1,1)$. So, the conditions from Theorem \ref{thmmobius} on $N$ are fulfilled. Furthermore, applying the Erd\'{e}lyi's asymptotic formula \cite[p. 148]{estrada-kanwal} for finite part integrals we easily see that the entire function 
$$
G(s)=-\mathrm{F.p.}\int_{0}^{e^{2}}e^{-(s-1)y}\frac{1+\cos y}{y}\:\mathrm{d}y+2\gamma
$$
 occurring in the formula for $\log\zeta(s)$ satisfies $G(1+it)=2 \log|t| +O(1)$ and that all of its derivatives $G^{(n)}(1+it)=o(1)$. So, $\zeta^{(n)}(1+it)\ll 1$ for each $n\in\mathbb{N}$. Therefore, we obtain that $E\ast K\in \mathcal{S}(\mathbb{R})$ for all $K\in\mathcal{S}(\mathbb{R})$ without any restriction on the support of its Fourier transform. In particular, $E\ast K\in L^{1}(\mathbb{R})$ for all kernels $K\in\mathcal{S}(\mathbb{R})$.

That $M(x)=o(x)$ can be verified here by applying Tauberian theorems. For instance, we have that $\mathrm{d}M+\mathrm{d}N=2\sum_{n=0}^{\infty}\mathrm{d}\Pi^{\ast_{M}2n}/(2n)!$ is a positive measure with $\mathcal{L}\{\mathrm{d}M+\mathrm{d}N;s\}=\zeta(s)+1/\zeta(s)$; by Theorem \ref{thtauberian2}, $M(x)+N(x)\sim ax$, i.e., $M(x)=o(x)$. Theorem \ref{thmmobius} implies $m(x)=o(1)$, but this can also be deduced from Theorem \ref{thtauberian1} because $m(e^{x})$ is slowly decreasing and in this example we control $\zeta(s)$ completely: $\mathcal{L}\{m(e^{x});s\}=1/(s\zeta(s+1))$ has continuous extension to $\Re e\:s=0$. 
 
One can also construct a discrete example sharing similar properties with $\Pi$ and $N$ via Diamond's discretization procedure \cite{d-s-v,diamond2}. In fact, define the generalized primes
$$
P=\{p_k\}_{k=1}^{\infty},\quad  p_{k}=\Pi^{-1}(k)
$$
and denote by $N_P$, $\pi_{P}$, $M_{P}$, $m_{P}$, and $\zeta_{P}$ the associated generalized number-theoretic functions. Since $\pi_{P}(x)=\Pi(x)+O(1)$, we have that $\zeta_{P}(s)/\zeta(s)$ analytically extends to $\Re e\: s>1/2$. By using the same arguments as for the continuous example,  we easily get (with $c=a\zeta_{P}(1)/\zeta(1)$) that
$$
M_{P}(x)=o(x), \quad m_{P}(x)=o(1), \quad
\int_{1}^{\infty}\left|\frac{N_{P}(x)-cx}{x^{2}}\right| \mathrm{d}x=\infty, \quad E_{P}\ast K\in L^{1}(\mathbb{R}),
$$
with $E_{P}(y)= e^{-y}N_{P}(e^{y})-c$ and any $K\in L^{1}(\mathbb{R})$ with $\operatorname*{supp}\hat{K}\subset (-1,1)$,
and
$$
 \pi_{P}(x)=\frac{x}{\log x}\left(1+\frac{\sqrt{2}}{2}\cos \left(\log x-\frac{\pi}{4}\right)\right)+O\left(\frac{x}{\log^{2} x}\right).
$$

It can be shown as well that there are $b_{0}\neq0,b_{1},\dots$ and $\beta_{0},\beta_{1},\dots$ such that
\begin{align}
\label{pnteeqex4.2}
N_{P}(x)&\sim cx+ \frac{x}{\log^{1/2}x}\sum^{\infty}_{j=0}b_{j}\frac{\cos(\log x+\beta_{j})}{\log^{j}x}
\\
\nonumber
&= cx+b_{0}\:\frac{x\cos(\log x+ \beta_{0})}{\log^{1/2} x}+ O\left(\frac{x}{\log^{3/2}x}\right).
\end{align}
The proofs of the asymptotic formulas (\ref{pnteeqex4.1}) and (\ref{pnteeqex4.2}) require additional work, but the details go along the same lines as those provided in \cite[Sect.~4 and Sect.~5]{d-s-v}; we therefore choose to omit them.
\end{example}

\begin{remark}\label{pnteex4} Diamond and Zhang used a simple example \cite{diamond-zhangbook} to show that in general the implication $M(x)=o(x)\Rightarrow m(x)=o(1)$ does not hold. They considered $\pi(x)= \sum_{p\leq x}p^{-1}$, where the sum runs over all rational primes. Here one has $N(x)=o(x)$ and $M(x)=o(x)$, but $m(x)=6/\pi^{2}+o(1)$. Presumably, pointwise asymptotics of type (\ref{thmainas1simplified})
could be unrelated to the conditional equivalence between $M(x)=o(x)$ and $m(x)=o(1)$. So, we wonder: Are there examples satisfying (\ref{thmainas1simplified}) but for which  $M(x)=o(x)$ does not imply $m(x)=o(1)$?
\end{remark}
\begin{remark} In analogy to Example \ref{pnteex3}, it would be interesting to construct an example of  a generalized number system such that the sharp Mertens relation holds, $N$ satisfies the conditions (\ref{thmainas1}) and (\ref{thmainas3}) from Theorem \ref{thmain} for suitable kernels $K_1$ and $K_2$, and such that (\ref{thmainas1simplified}) and (\ref{thmainas3simplified}) do not hold.
 
\end{remark}

\end{document}